\documentclass[11pt]{amsart}

\usepackage{amsmath}
\usepackage{amssymb}
\usepackage{graphicx}
\newtheorem{theorem}{Theorem}[section]
\newtheorem{corollary}[theorem]{Corollary}
\newtheorem{lemma}[theorem]{Lemma}

\theoremstyle{definition}

\newtheorem{remark}[theorem]{Remark}

\numberwithin{equation}{section}

\DeclareMathOperator{\wt}{wt}
\DeclareMathOperator{\dep}{dep}
\DeclareMathOperator{\htt}{ht}
\DeclareMathOperator{\Surj}{Surj}

\newcommand{\PPP}{\mathcal{P}}
\newcommand{\Q}{\mathbb{Q}}
\newcommand{\R}{\mathbb{R}}
\newcommand{\SSS}{\mathfrak{S}}
\newcommand{\Z}{\mathbb{Z}}
\newcommand{\ZZ}{\mathcal{Z}}
\title[Interpolated polynomial MZVs of fixed wt, dep, and ht]{Interpolated polynomial multiple zeta values of fixed weight, depth, and height}
\dedicatory{Dedicated to Professor David Preiss on the occasion of his 75th birthday.}
\author[M. Hirose]{Minoru Hirose}
\address[M. Hirose]{Institute for Advanced Research, Nagoya University,  Furo-cho, Chikusa-ku, Nagoya, 464-8602, Japan}
\email{minoru.hirose@math.nagoya-u.ac.jp}
\author[H. Murahara]{Hideki Murahara}
\address[H. Murahara]{The University of Kitakyushu,  4-2-1 Kitagata, Kokuraminami-ku, Kitakyushu, Fukuoka, 802-8577, Japan}
\email{hmurahara@mathformula.page}
\author[S. Saito]{Shingo Saito}
\address[S. Saito]{Faculty of Arts and Science, Kyushu University,  744, Motooka, Nishi-ku, Fukuoka, 819-0395, Japan}
\email{ssaito@artsci.kyushu-u.ac.jp}

\begin{document}
\begin{abstract}
 We define the interpolated polynomial multiple zeta values
 as a generalization of all of multiple zeta values, multiple zeta-star values,
 interpolated multiple zeta values, symmetric multiple zeta values, and polynomial multiple zeta values.
 We then compute the generating function of the sum of interpolated polynomial multiple zeta values of fixed weight, depth, and height.
\end{abstract}

\maketitle

%%%%%%%%%%%%%%%%%%%%%%%%%%%%%%%%%%%%%%%%%%%%%%%%%%%%%%%%%%%%%%%%%%%%%%%%%%%%%%%%

\section{Introduction}
\subsection{Multiple zeta values}
An \emph{index} is a finite (possibly empty) sequence of positive integers.
We say that an index is \emph{admissible} if either it is empty or its last component is greater than $1$.

For each admissible index $\boldsymbol{k}=(k_1,\dots,k_r)$,
we define the \emph{multiple zeta value} and \emph{multiple zeta-star value} by
\begin{align*}
 \zeta(\boldsymbol{k})&=\sum_{1\le n_1<\dots<n_r}\frac{1}{n_1^{k_1}\dotsm n_r^{k_r}}\in\R,\\
 \zeta^{\star}(\boldsymbol{k})&=\sum_{1\le n_1\le\dots\le n_r}\frac{1}{n_1^{k_1}\dotsm n_r^{k_r}}\in\R,
\end{align*}
respectively (note that the admissibility of $\boldsymbol{k}$ ensures the convergence of the infinite sums).
Let $\ZZ$ denote the $\Q$-vector space spanned by all the multiple zeta values (or equivalently by all the multiple zeta-star values).
It is known that the multiple zeta(-star) values satisfy a large number of linear and algebraic relations,
and our central interest lies in investigating such relations.

\subsection{Analogues and generalizations of multiple zeta values}
Following Yamamoto~\cite{Yam13}, we define for each admissible index $\boldsymbol{k}=(k_1,\dots,k_r)$ the \emph{interpolated multiple zeta value} by
\[
 \zeta^t(\boldsymbol{k})=\sum_{1\le n_1\le\dots\le n_r}\frac{t^{e(n_1,\dots,n_r)}}{n_1^{k_1}\dotsm n_r^{k_r}}\in\ZZ[t],
\]
where $e(n_1,\dots,n_r)$ denotes the number of $i=1,\dots,r-1$ for which $n_i=n_{i+1}$.
Note that $\zeta^0(\boldsymbol{k})=\zeta(\boldsymbol{k})$ and $\zeta^1(\boldsymbol{k})=\zeta^{\star}(\boldsymbol{k})$.

Following Kaneko and Zagier~\cite{KZ19},
we define for each index $\boldsymbol{k}=(k_1,\dots,k_r)$ the \emph{symmetric multiple zeta value} and \emph{symmetric multiple zeta-star value} by
\begin{align*}
 \zeta_S(\boldsymbol{k})&=\sum_{i=0}^{r}(-1)^{k_{i+1}+\dots+k_r}\zeta(k_1,\dots,k_i)\zeta(k_r,\dots,k_{i+1})\in\ZZ,\\
 \zeta_S^{\star}(\boldsymbol{k})&=\sum_{i=0}^{r}(-1)^{k_{i+1}+\dots+k_r}\zeta^{\star}(k_1,\dots,k_i)\zeta^{\star}(k_r,\dots,k_{i+1})\in\ZZ,
\end{align*}
where in the right-hand side we extend $\zeta$ and $\zeta^{\star}$ to non-admissible indices by using the harmonic regularization (see \cite{IKZ06} for further details),
so that $\boldsymbol{k}$ need not be admissible.
Note that some papers use the term ``the symmetric multiple zeta(-star) values'' to mean $\zeta_S^{(\star)}(\boldsymbol{k})\bmod\zeta(2)\ZZ$ in $\ZZ/\zeta(2)\ZZ$,
as opposed to our terminology.

Following our previous works \cite{HMS20} and \cite{HMS21},
we define for each index $\boldsymbol{k}=(k_1,\dots,k_r)$ the \emph{polynomial multiple zeta value} and \emph{polynomial multiple zeta-star value} by
\begin{align*}
 \zeta_{x,y}(\boldsymbol{k})&=\sum_{i=0}^{r}\zeta(k_1,\dots,k_i)\zeta(k_r,\dots,k_{i+1})x^{k_1+\dots+k_i}y^{k_{i+1}+\dots+k_r}\in\ZZ[x,y],\\
 \zeta_{x,y}^{\star}(\boldsymbol{k})&=\sum_{i=0}^{r}\zeta^{\star}(k_1,\dots,k_i)\zeta^{\star}(k_r,\dots,k_{i+1})x^{k_1+\dots+k_i}y^{k_{i+1}+\dots+k_r}\in\ZZ[x,y].
\end{align*}
Note that $\zeta_{1,0}(\boldsymbol{k})=\zeta(\boldsymbol{k})$, $\zeta_{1,0}^{\star}(\boldsymbol{k})=\zeta^{\star}(\boldsymbol{k})$,
$\zeta_{1,-1}(\boldsymbol{k})=\zeta_S(\boldsymbol{k})$, and $\zeta_{1,-1}^{\star}(\boldsymbol{k})=\zeta_S^{\star}(\boldsymbol{k})$.

Finally we define for each index $\boldsymbol{k}=(k_1,\dots,k_r)$ the \emph{interpolated polynomial multiple zeta value} by
\[
 \zeta_{x,y}^t(\boldsymbol{k})=\sum_{i=0}^{r}\zeta^t(k_1,\dots,k_i)\zeta^t(k_r,\dots,k_{i+1})x^{k_1+\dots+k_i}y^{k_{i+1}+\dots+k_r}\in\ZZ[t,x,y].
\]
Note that $\zeta_{x,y}^0(\boldsymbol{k})=\zeta_{x,y}(\boldsymbol{k})$, $\zeta_{x,y}^1(\boldsymbol{k})=\zeta_{x,y}^{\star}(\boldsymbol{k})$,
and $\zeta_{1,0}^t(\boldsymbol{k})=\zeta^t(\boldsymbol{k})$,
so that the interpolated polynomial multiple zeta values generalize all values defined above.

\subsection{Ohno-Zagier theorem and its generalization}
For each index $\boldsymbol{k}=(k_1,\dots,k_r)$,
we define its \emph{weight} $\wt(\boldsymbol{k})$, \emph{depth} $\dep(\boldsymbol{k})$, and \emph{height} $\htt(\boldsymbol{k})$ by
\[
 \wt(\boldsymbol{k})=k_1+\dots+k_r,\quad
 \dep(\boldsymbol{k})=r,\quad
 \htt(\boldsymbol{k})=\#\{i=1,\dots,r\mid k_i\ge2\},
\]
and we put
\[
 u_{\boldsymbol{k}}=X^{\wt(\boldsymbol{k})-\dep(\boldsymbol{k})-\htt(\boldsymbol{k})}Y^{\dep(\boldsymbol{k})-\htt(\boldsymbol{k})}Z^{\htt(\boldsymbol{k})}
 \in\Q[X,Y,Z].
\]
Note that writing $u_k=u_{(k)}$ for every positive integer $k$, we have $u_{(k_1,\dots,k_r)}=u_{k_1}\dotsm u_{k_r}$ for every index $(k_1,\dots,k_r)$.

Let $I$ denote the set of all nonempty admissible indices.
Ohno and Zagier~\cite{OZ01} computed the generating function
\[
 \Phi(X,Y,Z)=\sum_{\boldsymbol{k}\in I}\zeta(\boldsymbol{k})u_{\boldsymbol{k}}\in\ZZ[[X,Y,Z]]
\]
and obtained the following formula:
\begin{theorem}[Ohno and Zagier \cite{OZ01}]\label{theorem:OZ}
 We have 
 \begin{align*}
  \Phi(X,Y,Z)
  &=\frac{Z}{XY-Z}\biggl(1-\frac{\Gamma(1-X)\Gamma(1-Y)}{\Gamma(1-\alpha)\Gamma(1-\beta)}\biggr)\\
  &=\frac{Z}{XY-Z}\Biggl(1-\exp\Biggl(\sum_{k=2}^{\infty}\frac{\zeta(k)}{k}(X^k+Y^k-\alpha^k-\beta^k)\Biggr)\Biggr),
 \end{align*}
 where $\alpha$ and $\beta$ satisfy $\alpha+\beta=X+Y$ and $\alpha\beta=Z$.
\end{theorem}

Aoki, Kombu, and Ohno~\cite{AKO08} obtained an analogous theorem by computing the generating function
\[
 \Phi^{\star}(X,Y,Z)=\sum_{\boldsymbol{k}\in I}(-1)^{\dep(\boldsymbol{k})}\zeta^{\star}(\boldsymbol{k})u_{\boldsymbol{k}}\in\ZZ[[X,Y,Z]].
\]
Furthermore, Li and Qun~\cite{LQ17} computed the generating function
\[
 \Phi^t(X,Y,Z)=\sum_{\boldsymbol{k}\in I}(1-2t)^{\dep(\boldsymbol{k})}\zeta^t(\boldsymbol{k})u_{\boldsymbol{k}}\in\ZZ[t][[X,Y,Z]]
\]
and obtained the following formula:
\begin{theorem}[Li and Qun~\cite{LQ17}]\label{theorem:LQ}
 We have
 \[
  \Phi^t(X,Y,Z)=\frac{Z}{(1-Y)(1-\beta_t)}{}_3F_2\biggl(\begin{matrix}1+\alpha_{t-1}-\beta_t,1+\beta_{t-1}-\beta_t,1\\2-Y,2-\beta_t\end{matrix};1\biggr),
 \]
 where $\alpha_s$ and $\beta_s$ satisfy $\alpha_s+\beta_s=X+sY$ and $\alpha_s\beta_s=s(XY-Z)$.
\end{theorem}

\subsection{Statement of our main theorem}
In this paper, we compute the generating function
\[
 \Phi_{x,y}^t(X,Y,Z)=\sum_{\boldsymbol{k}\in I}(1-2t)^{\dep(\boldsymbol{k})}\zeta_{x,y}^t(\boldsymbol{k})u_{\boldsymbol{k}}.
\]
and prove the following theorem:
\begin{theorem}[Main theorem]\label{theorem:main1}
 We have
 \begin{align*}
  &\Phi_{x,y}^t(X,Y,Z)\\
  &=\Phi^t(xX,xY,x^2Z)\\
  &\phantom{{}={}}-\Phi^{1-t}(yX,yY,y^2Z)\exp\Biggl(\sum_{k=2}^{\infty}\frac{\zeta(k)}{k}(x^k+y^k)(\gamma_t^k+\delta_t^k-\gamma_{1-t}^k-\delta_{1-t}^k)\Biggr),
 \end{align*}
 where $\gamma_t=\alpha_{t(1-2t)}$ and $\delta_t=\beta_{t(1-2t)}$.
\end{theorem}

With the aid of Theorem~\ref{theorem:LQ}, our main theorem allows us to compute $\Phi_{x,y}^t(X,Y,Z)$ completely explicity.

\section{Proof of our main theorem}
In this section, we give a proof of our main theorem (Theorem~\ref{theorem:main1})
by showing the following two lemmas:
\begin{lemma}\label{lemma:Phi^t-Phi_x,y^t}
 We have
 \begin{align*}
  &\Phi^t(xX,xY,x^2Z)-\Phi_{x,y}^{t}(X,Y,Z)\\
  &=\Phi^{1-t}(yX,yY,y^2Z)\sum_{\boldsymbol{k}}(1-2t)^{\dep(\boldsymbol{k})}\zeta_{x,y}^{t}(\boldsymbol{k})u_{\boldsymbol{k}}.
 \end{align*}
\end{lemma}

\begin{lemma}\label{lemma:symgene}
 We have 
 \begin{align*}
  &\sum_{\boldsymbol{k}}(1-2t)^{\dep(\boldsymbol{k})}\zeta_{x,y}^{t}(\boldsymbol{k})u_{\boldsymbol{k}}\\
  &=\exp\Biggl(\sum_{k=2}^{\infty}\frac{\zeta(k)}{k}(x^k+y^k)(\gamma_t^k+\delta_t^k-\gamma_{1-t}^{k}-\delta_{1-t}^{k})\Biggr).
 \end{align*}
\end{lemma}

Here and in the sequel, $\sum_{\boldsymbol{k}}$ denotes a sum over all indices $\boldsymbol{k}$.

\subsection{Proof of Lemma~\ref{lemma:Phi^t-Phi_x,y^t}}
\begin{lemma}\label{lemma:antipode_t}
 Let $(k_1,\dots,k_r)$ be a nonempty index.
 Then we have
 \[
  \sum_{j=0}^{r}(-1)^{r-j}\zeta^t(k_j,\dots,k_1)\zeta^{1-t}(k_{j+1},\dots,k_r)=0.
 \]
\end{lemma}

\begin{proof}
 This is essentially the same as the formula given in {\cite[Proposition 3.9]{Yam13}}.
\end{proof}

\begin{proof}[Proof of Lemma~\ref{lemma:Phi^t-Phi_x,y^t}]
 Let $A$ denote the right-hand side of the desired equality.
 Since $u_{\boldsymbol{k}}$ with $(X,Y,Z)$ replaced by $(yX,yY,y^2Z)$ is equal to $y^{\wt(\boldsymbol{k})}u_{\boldsymbol{k}}$ for every index $\boldsymbol{k}$,
 we have
 \begin{align*}
  A&=\Biggl(\sum_{\boldsymbol{k}_2\in I}(1-2(1-t))^{\dep(\boldsymbol{k}_2)}\zeta^{1-t}(\boldsymbol{k}_2)y^{\wt(\boldsymbol{k}_2)}u_{\boldsymbol{k}_2}\Biggr)\\
   &\phantom{{}={}}\times\Biggl(\sum_{\boldsymbol{k}_1}(1-2t)^{\dep(\boldsymbol{k}_1)}\zeta_{x,y}^{t}(\boldsymbol{k}_1)u_{\boldsymbol{k}_1}\Biggr)\\
  &=\sum_{\boldsymbol{k}_1}\sum_{\boldsymbol{k}_2\in I}(-1)^{\dep(\boldsymbol{k}_2)}(1-2t)^{\dep(\boldsymbol{k}_1,\boldsymbol{k}_2)}\zeta_{x,y}^{t}(\boldsymbol{k}_1)\zeta^{1-t}(\boldsymbol{k}_2)y^{\wt(\boldsymbol{k}_2)}u_{(\boldsymbol{k}_1,\boldsymbol{k}_2)}\\
  &=\sum_{\boldsymbol{k}=(k_1,\dots,k_r)\in I}(1-2t)^{\dep(\boldsymbol{k})}\\
  &\phantom{{}={}}\times\Biggl(\sum_{j=0}^{r-1}(-1)^{r-j}\zeta_{x,y}^{t}(k_1,\dots,k_j)\zeta^{1-t}(k_{j+1},\dots,k_r)y^{k_{j+1}+\dots+k_r}\Biggr)u_{\boldsymbol{k}},
 \end{align*}
 where $(\boldsymbol{k}_1,\boldsymbol{k}_2)$ denotes the concatenation of the indices $\boldsymbol{k}_1$ and $\boldsymbol{k}_2$.
 Here for each $\boldsymbol{k}=(k_1,\dots,k_r)\in I$, we have
 \begin{align*}
  &\sum_{j=0}^{r-1}(-1)^{r-j}\zeta_{x,y}^{t}(k_1,\dots,k_j)\zeta^{1-t}(k_{j+1},\dots,k_r)y^{k_{j+1}+\dots+k_r}\\
  &=\sum_{j=0}^{r-1}(-1)^{r-j}\Biggl(\sum_{i=0}^{j}\zeta^t(k_1,\dots,k_i)\zeta^t(k_j,\dots,k_{i+1})x^{k_1+\dots+k_i}y^{k_{i+1}+\dots+k_r}\Biggr)\\
  &\hphantom{{}=\sum_{j=0}^{r-1}}\times\zeta^{1-t}(k_{j+1},\dots,k_r)\\
  &=\sum_{i=0}^{r-1}\zeta^t(k_1,\dots,k_i)\Biggl(\sum_{j=i}^{r-1}(-1)^{r-j}\zeta^t(k_j,\dots,k_{i+1})\zeta^{1-t}(k_{j+1},\dots,k_r)\Biggr)\\
  &\hphantom{{}=\sum_{i=0}^{r-1}}\times x^{k_1+\dots+k_i}y^{k_{i+1}+\dots+k_r}\\
  &=-\sum_{i=0}^{r-1}\zeta^t(k_1,\dots,k_i)\zeta^{t}(k_r,\dots,k_{i+1})x^{k_1+\dots+k_i}y^{k_{i+1}+\dots+k_r}\\
  &=\zeta^t(\boldsymbol{k})x^{\wt(\boldsymbol{k})}-\zeta_{x,y}^{t}(\boldsymbol{k}),
 \end{align*}
 where the third equality follows from Lemma~\ref{lemma:antipode_t}.
 Therefore multiplying by $(1-2t)^{\dep(\boldsymbol{k})}u_{\boldsymbol{k}}$ and summing over all $\boldsymbol{k}\in I$, we obtain
 \begin{align*}
  A&=\sum_{\boldsymbol{k}\in I}(1-2t)^{\dep(\boldsymbol{k})}(\zeta^t(\boldsymbol{k})x^{\wt(\boldsymbol{k})}-\zeta_{x,y}^{t}(\boldsymbol{k}))u_{\boldsymbol{k}}\\
   &=\Phi^t(xX,xY,x^2Z)-\Phi_{x,y}^{t}(X,Y,Z),
 \end{align*}
 as required.
\end{proof}

\subsection{Proof of Lemma~\ref{lemma:symgene}}
We write $\lvert P\rvert$ for the cardinality of a finite set $P$.
For each positive integer $r$, we write $\PPP_r$ for the set of all partitions of the set $\{1,\dots,r\}$
(recall that a \emph{partition} of a set $S$ is a family of nonempty disjoint subsets of $S$ whose union is $S$),
and we put $c_r(t)=(r-1)!(t^r-(t-1)^r)\in\Z[t]$.

\begin{lemma}\label{lemma:sym_sum}
 Let $r$ be a positive integer and $(k_1,\dots,k_r)$ be an index.
 Then we have
 \[
  \sum_{\sigma\in\SSS_r}\zeta_{x,y}^{t}(k_{\sigma(1)},\dots,k_{\sigma(r)})
  =\sum_{\Pi\in\PPP_r}\prod_{P\in\Pi}c_{\lvert P\rvert}(t)\zeta_{x,y}\Biggl(\sum_{i\in P}k_i\Biggr),
 \]
 where $\SSS_r$ denotes the symmetric group of degree $r$.
\end{lemma}

\begin{proof}
 The lemma follows from \cite[Theorem 13]{Hof20}; see \cite[Equation~(8.4)]{Hof20} for further details and
 note that $\zeta^t$ and $\zeta_{x,y}^t$ enjoy the same $t$-harmonic product rule.
\end{proof}

\begin{proof}[Proof of Lemma~\ref{lemma:symgene}]
 Lemma~\ref{lemma:sym_sum} implies that
 \begin{align*}
  &\sum_{\boldsymbol{k}\ne\emptyset}(1-2t)^{\dep(\boldsymbol{k})}\zeta_{x,y}^{t}(\boldsymbol{k})u_{\boldsymbol{k}}\\
  &=\sum_{r=1}^{\infty}(1-2t)^r\sum_{k_1,\dots,k_r\ge1}\zeta_{x,y}^{t}(k_1,\dots,k_r)u_{(k_1,\dots,k_r)}\\
  &=\sum_{r=1}^{\infty}\frac{(1-2t)^r}{r!}\sum_{\substack{k_1,\dots,k_r\ge1\\\sigma\in\SSS_r}}\zeta_{x,y}^{t}(k_{\sigma(1)},\dots,k_{\sigma(r)})u_{k_1}\dotsm u_{k_r}\\
  &=\sum_{r=1}^{\infty}\frac{(1-2t)^r}{r!}\sum_{\substack{k_1,\dots,k_r\ge1\\\Pi\in\PPP_r}}\prod_{P\in\Pi}\Biggl(c_{\lvert P\rvert}(t)\zeta_{x,y}\Biggl(\sum_{i\in P}k_i\Biggr)\prod_{i\in P}u_{k_i}\Biggr).
 \end{align*}
 Now for each positive integer $r$, we have
 \begin{align*}
  &\sum_{\substack{k_1,\dots,k_r\ge1\\\Pi\in\PPP_r}}\prod_{P\in\Pi}\Biggl(c_{\lvert P\rvert}(t)\zeta_{x,y}\Biggl(\sum_{i\in P}k_i\Biggr)\prod_{i\in P}u_{k_i}\Biggr)\\
  &=\sum_{\Pi\in\PPP_r}\prod_{P\in\Pi}\Biggl(c_{\lvert P\rvert}(t)\sum_{\dep(\boldsymbol{k})=\lvert P\rvert}\zeta_{x,y}(\wt(\boldsymbol{k}))u_{\boldsymbol{k}}\Biggr)\\
  &=\sum_{l=1}^{r}\sum_{\substack{\Pi\in\PPP_r\\\lvert\Pi\rvert=l}}\prod_{P\in\Pi}\Biggl(c_{\lvert P\rvert}(t)\sum_{\dep(\boldsymbol{k})=\lvert P\rvert}\zeta_{x,y}(\wt(\boldsymbol{k}))u_{\boldsymbol{k}}\Biggr)\\
  &=\sum_{l=1}^{r}\frac{1}{l!}\sum_{f\in\Surj(r,l)}\prod_{j=1}^{l}\Biggl(c_{\lvert f^{-1}(j)\rvert}(t)\sum_{\dep(\boldsymbol{k})=\lvert f^{-1}(j)\rvert}\zeta_{x,y}(\wt(\boldsymbol{k}))u_{\boldsymbol{k}}\Biggr)\\
  &=\sum_{l=1}^{r}\frac{1}{l!}\sum_{\substack{r_1,\dots,r_l\ge1\\r_1+\dots+r_l=r}}\frac{r!}{r_1!\dotsm r_l!}\prod_{j=1}^{l}\Biggl(c_{r_j}(t)\sum_{\dep(\boldsymbol{k})=r_j}\zeta_{x,y}(\wt(\boldsymbol{k}))u_{\boldsymbol{k}}\Biggr),
 \end{align*}
 where $\Surj(r,l)$ denotes the set of all surjections from the set $\{1,\dots,r\}$ to the set $\{1,\dots,l\}$.
 It follows that
 \begin{align*}
  &\sum_{\boldsymbol{k}\ne\emptyset}(1-2t)^{\dep(\boldsymbol{k})}\zeta_{x,y}^{t}(\boldsymbol{k})u_{\boldsymbol{k}}\\
  &=\sum_{r=1}^{\infty}\frac{(1-2t)^r}{r!}\sum_{l=1}^{r}\frac{1}{l!}\sum_{\substack{r_1,\dots,r_l\ge1\\r_1+\dots+r_l=r}}\frac{r!}{r_1!\dotsm r_l!}\prod_{j=1}^{l}\Biggl(c_{r_j}(t)\sum_{\dep(\boldsymbol{k})=r_j}\zeta_{x,y}(\wt(\boldsymbol{k}))u_{\boldsymbol{k}}\Biggr)\\
  &=\sum_{l=1}^{\infty}\frac{1}{l!}\sum_{r_1,\dots,r_l\ge1}\prod_{j=1}^{l}\Biggl(\frac{(1-2t)^{r_j}}{r_j!}c_{r_j}(t)\sum_{\dep(\boldsymbol{k})=r_j}\zeta_{x,y}(\wt(\boldsymbol{k}))u_{\boldsymbol{k}}\Biggr)\\
  &=\sum_{l=1}^{\infty}\frac{1}{l!}\Biggl(\sum_{r=1}^{\infty}\frac{(1-2t)^r}{r!}c_r(t)\sum_{\dep(\boldsymbol{k})=r}\zeta_{x,y}(\wt(\boldsymbol{k}))u_{\boldsymbol{k}}\Biggr)^l\\
  &=\exp\Biggl(\sum_{r=1}^{\infty}\frac{(1-2t)^r}{r!}c_r(t)\sum_{\dep(\boldsymbol{k})=r}\zeta_{x,y}(\wt(\boldsymbol{k}))u_{\boldsymbol{k}}\Biggr)-1.
 \end{align*}

 Therefore it suffices to show that
 \begin{align*}
  &\sum_{r=1}^{\infty}\frac{(1-2t)^r}{r!}c_r(t)\sum_{\dep(\boldsymbol{k})=r}\zeta_{x,y}(\wt(\boldsymbol{k}))u_{\boldsymbol{k}}\\
  &=\sum_{k=2}^{\infty}\frac{\zeta(k)}{k}(x^k+y^k)(\gamma_t^k+\delta_t^k-\gamma_{1-t}^{k}-\delta_{1-t}^{k}).
 \end{align*}
 Since in $\Q[X,Y,Z][[T]]$ we have
 \begin{align*}
  &\sum_{k=1}^{\infty}\Biggl(\sum_{r=1}^{\infty}\frac{(1-2t)^r}{r!}c_r(t)\sum_{\substack{\dep(\boldsymbol{k})=r\\\wt(\boldsymbol{k})=k}}u_{\boldsymbol{k}}\Biggr)T^k\\
  &=\sum_{r=1}^{\infty}\frac{(1-2t)^r}{r!}c_r(t)\sum_{k_1,\dots,k_r\ge1}u_{k_1}\dotsm u_{k_r}T^{k_1+\dots+k_r}\\
  &=\sum_{r=1}^{\infty}\frac{(1-2t)^r}{r!}c_r(t)\Biggl(\sum_{k=1}^{\infty}u_kT^k\Biggr)^r\\
  &=\sum_{r=1}^{\infty}\frac{(t(1-2t))^r-((1-t)(2t-1))^r}{r}\Biggl(YT+\frac{ZT^2}{1-XT}\Biggr)^r\\
  &=-\log\frac{1-t(1-2t)\bigl(YT+\frac{ZT^2}{1-XT}\bigr)}{1-(1-t)(2t-1)\bigl(YT+\frac{ZT^2}{1-XT}\bigr)}\\
  &=-\log\frac{1-(X+t(1-2t)Y)T+t(1-2t)(XY-Z)T^2}{1-(X+(1-t)(2t-1)Y)T+(1-t)(2t-1)(XY-Z)T^2}\\
  &=-\log\frac{(1-\gamma_tT)(1-\delta_tT)}{(1-\gamma_{1-t}T)(1-\delta_{1-t}T)}\\
  &=\sum_{k=1}^{\infty}\frac{\gamma_t^k+\delta_t^k-\gamma_{1-t}^{k}-\delta_{1-t}^{k}}{k}T^k,
 \end{align*}
 we obtain
 \begin{align*}
  &\sum_{r=1}^{\infty}\frac{(1-2t)^r}{r!}c_r(t)\sum_{\dep(\boldsymbol{k})=r}\zeta_{x,y}(\wt(\boldsymbol{k}))u_{\boldsymbol{k}}\\
  &=\sum_{k=1}^{\infty}\Biggl(\sum_{r=1}^{\infty}\frac{(1-2t)^r}{r!}c_r(t)\sum_{\substack{\dep(\boldsymbol{k})=r\\\wt(\boldsymbol{k})=k}}u_{\boldsymbol{k}}\Biggr)\zeta_{x,y}(k)\\
  &=\sum_{k=1}^{\infty}\frac{\gamma_t^k+\delta_t^k-\gamma_{1-t}^{k}-\delta_{1-t}^{k}}{k}\zeta_{x,y}(k)\\
  &=\sum_{k=1}^{\infty}\frac{\zeta(k)}{k}(x^k+y^k)(\gamma_t^k+\delta_t^k-\gamma_{1-t}^{k}-\delta_{1-t}^{k}),
 \end{align*}
 as required.
\end{proof}

\section{Corollaries}
We write $\Phi_S^t(X,Y,Z)=\Phi_{1,-1}^t(X,Y,Z)$ and $\Phi_S^{\star}(X,Y,Z)=\Phi_S^1(X,Y,Z)$.

\begin{corollary}\label{corollary:mod_zeta2}
 We have
 \[
  \Phi_S^t(X,Y,Z)\equiv\Phi^t(X,Y,Z)-\Phi^{1-t}(-X,-Y,Z)\pmod{\zeta(2)\ZZ}.
 \]
\end{corollary}

\begin{proof}
 Set $x=1$ and $y=-1$ in Theorem~\ref{theorem:main1}, and observe that $\zeta(k)$ is a multiple of $\zeta(2)$ whenever $k$ is a even positive integer.
\end{proof}

\begin{corollary}
 We have
 \[
  \sum_{\boldsymbol{k}\in I}(\zeta_{S}^{\star}(\boldsymbol{k})-\zeta^{\star}(\boldsymbol{k}))u_{\boldsymbol{k}}
  =\frac{Z}{XY-Z}\biggl(\frac{\Gamma(1+X)\Gamma(1-Y)}{\Gamma(1-\eta)\Gamma(1-\gamma)}-1\biggr)\pmod{\zeta(2)\ZZ},
 \]
 where $\eta$ and $\xi$ satisfy $\eta+\xi=-X+Y$ and $\eta\xi=-Z$.
\end{corollary}

\begin{proof}
 Since $u_{\boldsymbol{k}}$ with $(X,Y,Z)$ replaced by $(X,-Y,-Z)$ is equal to $(-1)^{\dep(\boldsymbol{k})}u_{\boldsymbol{k}}$, we have
 \begin{align*}
  \sum_{\boldsymbol{k}\in I}(\zeta_S^{\star}(\boldsymbol{k})-\zeta^{\star}(\boldsymbol{k}))u_{\boldsymbol{k}}
  &=\Phi_S^{\star}(X,-Y,-Z)-\Phi^{\star}(X,-Y,-Z)\\
  &\equiv-\Phi(-X,Y,-Z)\pmod{\zeta(2)\ZZ}
 \end{align*}
 by Corollary~\ref{corollary:mod_zeta2} with $t=1$.
 Therefore the corollary follows from Theorem~\ref{theorem:OZ}.
\end{proof}

\begin{remark}
 Fujita and Komori \cite{FK21} computed the same sum
 \[
  \sum_{\boldsymbol{k}\in I}(\zeta_{S}^{\star}(\boldsymbol{k})-\zeta^{\star}(\boldsymbol{k}))u_{\boldsymbol{k}}
 \]
 modulo $\zeta(2)\ZZ$ using differential equations, in contrast to the algebraic nature of our proof.
\end{remark}

\section*{Acknowledgements}
The authors would like to thank Professor Yasushi Komori and Kento Fujita for sending us their paper before its publication.
This work was supported by JSPS KAKENHI Grant Numbers JP18J00982, JP18K03243, JP18K13392, JP22K03244, and JP22K13897,
by Grant for Basic Science Research Projects from The Sumitomo Foundation, and by Research Funding Granted by The University of Kitakyushu.

\end{document}